\documentclass[12pt]{article}

\usepackage{amssymb}
\usepackage{amsmath}
\usepackage{amsthm}
\usepackage[showonlyrefs]{mathtools}
\mathtoolsset{showonlyrefs}
\usepackage[breaklinks=true,colorlinks, linkcolor=blue, anchorcolor=blue, citecolor=blue]{hyperref}

\usepackage[dvips]{graphicx}
\usepackage{psfrag}
\usepackage{amscd}
\usepackage[latin1]{inputenc}

\numberwithin{equation}{section}

\newcommand{\comm}[1]{}

\def\diam{\operatorname{diam}}
\def\dist{\operatorname{dist}}

\def\crit{\operatorname{Crit}}
\def\jrit{\operatorname{Jrit}}

\def\({\left(}
\def\){\right)}
\def\oli{\overline}

\def\raw{\rightarrow}

\def\fa{~\,~\text{for all}~\,}

\def\no={\neq}
\def\sm{\setminus}

\def\C{{\mathbb C}}

\def\P{{\mathbb P}}

\def\BB{{\mathcal B}}

\def\EE{{\mathcal E}}
\def\FF{{\mathcal F}}

\def\JJ{{\mathcal J}}

\def\NN{{\mathcal N}}
\def\OO{{\mathcal O}}

\def\QQ{{\mathcal Q}}
\def\RR{{\mathcal R}}

\def\al{\alpha}

\def\ga{\gamma}

\def\vep{\varepsilon}

\def\io{\iota}

\def\la{\lambda}

\def\De{\Delta}

\def\La{\Lambda}

\def\Om{\Omega}

\theoremstyle{plain}

\newtheorem{Thm}{Theorem}[section]

\newtheorem{Lem}[Thm]{Lemma}

\newtheorem{Cor}[Thm]{Corollary}

\newtheorem*{Ack}{Acknowledgement}
\newtheorem{Def}[Thm]{Definition}

\theoremstyle{remark}
\newtheorem{Rem}[Thm]{Remark}

\begin{document}

\title{Slowly recurrent Collet--Eckmann maps with non-empty Fatou set}

\author{Magnus Aspenberg, Mats Bylund and Weiwei Cui}

\date{}

\maketitle

\begin{abstract}
In this paper we study rational Collet--Eckmann maps for which the Julia set is not the whole sphere and for which the critical points are recurrent at a slow rate. In families where the orders of the critical points are fixed, we prove that such maps are Lebesgue density points of hyperbolic maps. In particular, if all critical points are simple, they are Lebesgue density points of hyperbolic maps in the full space of rational maps of any degree $d \geq 2$. 
  \end{abstract}

 \section{Introduction}

Uniformly expanding maps have the property that nearby points on the Julia set repel each other at a uniform rate (with respect to some smooth metric). One of the central problems in complex dynamics is to prove that the set of these so called hyperbolic maps is open and dense in the parameter space of rational maps (or other complex analytic families of maps). This \emph{conjecture by P. Fatou} in the 1920s has been proven in the real case \cite{GSW, ML2, SSK}, but is still open in the complex setting.  In recent years,  a great deal of attention has been focused on maps which are non-hyperbolic but satisfy a certain non-uniformly expanding condition, like the Misiurewicz condition (critically non-recurrent or even postcritically finite maps),  the Collet--Eckmann condition or other more general summability conditions, semi-hyperbolicity etc. Conjecturally, almost every map is hyperbolic or satisfies such a non-uniformly expanding condition. This would also imply that the Fatou conjecture is true. In this paper, we focus on maps which satisfy the Collet--Eckmann condition. Our result demonstrates that any such map, for which the critical points are allowed to be recurrent at a slow rate (slowly recurrent maps), can be perturbed into hyperbolic maps in a strong sense; they are Lebesgue density points of hyperbolic maps. We discuss the (rather weak) condition on slow recurrence more below.

The Collet--Eckmann condition stems from the pioneering works by P.~Collet and J.-P.~Eckmann in the 1980s \cite{Collet-Eckmann1}. In the real setting, there are many works on the perturbation of such maps, see e.g. the seminal papers \cite{BC1, BC2} by M.~Benedicks and L.~Carleson. M.~Tsujii generalised these results for real maps in \cite{Tsuji}, see also the more recent work of B.~Gao and W.~Shen \cite{Gao-Shen}. We are going to study perturbations of such maps in the complex setting. For the quadratic family and other unicritical families, J.~Graczyk and G.~\'{S}wi\c{a}tek recently made an extensive study of perturbations of typical Collet--Eckmann maps with respect to harmonic measure, in a series of papers \cite{GSW-Fine, GSW-Struc, GSW-Harmonic, GSW-Logd}. M.~Benedicks and J.~Graczyk also have an unpublished work on perturbations on such (quadratic or, more generally, unicritical) maps. The maps there and in the recent papers \cite{GSW-Fine, GSW-Struc} are also slowly recurrent, and hence the results in this paper is partially a generalisation of some of those results. We will not use harmonic measure, but develop the classical Benedicks--Carleson parameter exclusion techniques and combining it with  strong results on transversality, by G. Levin \cite{Levin-book}. Technically, this paper is closely related to \cite{MA7}.

Let $f$ be a rational map. As usual let $\JJ(f)$ and $\FF(f)$ denote the Julia and Fatou set of $f$ respectively. Let $\crit(f)$ be the set of critical points of $f$, i.e. the set of points with vanishing spherical derivative. With $\jrit(f)$  we mean the the set of critical points of $f$ contained in the Julia set, i.e. $\jrit(f) = \crit(f) \cap \JJ(f)$. As is standard, we let $f^n$ denote the $n$-th iterate of $f$.

In this paper we will consider perturbations of rational maps satisfying the following two properties. Recall that a rational map is called \emph{hyperbolic} if it is expanding on the Julia set or, equivalently, if every critical point belongs to the Fatou set and is attracted to an attracting cycle. If a rational map is not hyperbolic, it is called \emph{non-hyperbolic}. Derivatives are always with respect to the spherical metric on the Riemann sphere. 
 
 \begin{Def}[Collet--Eckmann condition]\label{CE}
A non-hyperbolic rational map $f$ without parabolic periodic points satisfies the \emph{Collet--Eckmann condition}, if there exist $C>0$ and $\gamma>0$ such that for each critical point $c$ in the Julia set of $f$, one has
\[
\vert Df^{n}(f(c))\vert \geq Ce^{\gamma n}~\,~\text{for all}~\, n\geq 0.
\]
 \end{Def}
 
 We will often refer to the constant $\gamma$ appearing in the above definition as the \emph{Lyapunov exponent} or simply the \emph{exponent}. Notice that the Collet--Eckmann condition is equivalent to requiring the lower Lyapunov exponent at critical values (in the Julia set) to be strictly positive.

  \begin{Def}[Slow recurrence]\label{SL}
A point $z$ is said to be \emph{slowly recurrent} if for any $\alpha>0$, there exists $K>0$ such that
\begin{equation} \label{slow-recurrence}
\dist(f^{n}(z), \jrit(f))\geq Ke^{-\alpha n}~\,~\text{for all}~\,n\geq 0.
\end{equation}
Moreover, we say that $f$ is slowly recurrent if every point in $\jrit(f)$ is slowly recurrent.
 \end{Def}

 This condition is conjecturally generic, as for example in the real quadratic family \cite{AM}. In fact, every Collet--Eckmann map has {\em some} $\al > 0$ for which \eqref{slow-recurrence} holds; see \cite[Lemma 2.2 or Lemma 2.3]{De-Pr-Ur}.

 We denote by $\RR_d$,  the space of rational maps of degree $d$. In this paper we always assume that $d \geq 2$. If we write $f(z) = p(z)/q(z)$, where $p$ and $q$ are polynomials, and the maximal degree of $p$ and $q$ is $d$, there are two local charts on the coefficients; one for the case when $\deg(p)=d$ and another for $\deg(q) = d$.  The Lebesgue measure on each of these charts are not equal but mutually absolutely continuous. So talking about sets of positive measure is independent of the chart used. We also mention that the Fubini-Study metric on $\RR_d$ (which is a measure on the projective space $\P^{2d+1}(\C)$) is mutually absolutely continuous with respect to the Lebesgue measure on each chart.

We will also consider a certain normalisation of the space $\RR_d$ as follows, following G. Levin \cite{Levin-book, Levin-multipliers}. We say that two maps $f$ and $g$ are equivalent if they are conjugate by a M\"obius transformation. Then we can consider the space $\La_{d,\oli{p'}} \subset \RR_d$, (see \cite{Levin-book}) up to equivalence, as the set  of rational maps $f$ of degree $d \geq 2$ with precisely $p'$ critical points, i.e. $\crit = \{ c_1, \ldots, c_{p'} \}$, with corresponding multiplicities $\oli{p'} = \{m_1, \ldots, m_{p'} \}$ (in the same order). This means in particular that all critical points move analytically inside $\La_{d,\oli{p'}}$. 

 We will prove the following result.

 \begin{Thm} \label{main1}
Any slowly recurrent rational Collet--Eckmann map $f \in \La_{d,\oli{p'}}$ of degree $d \geq 2$, for which the Julia set is not the entire sphere, is a Lebesgue density point of hyperbolic maps in $\La_{d,\oli{p'}}$.
 \end{Thm}

 If all critical points are simple, then $\La_{d,\oli{p'}}$ is locally equal to $\RR_d$ (up to M\"obius conjugacy), and we immediately get the following corollary.

 \begin{Cor}
   Any slowly recurrent rational Collet--Eckmann map $f$ of degree $d \geq 2$ with only simple critical points, and for which the Julia set is not the entire sphere, is a Lebesgue density point of hyperbolic maps in $\RR_d$.
 \end{Cor}
 
Note that if $f$ is Collet--Eckmann and $\FF(f)\neq\emptyset$, then the Fatou set $\FF(f)$ consists only of attracting cycles and the Julia set of $f$ has Lebesgue measure zero (and actually the Hausdorff dimension is strictly less than $2$) \cite{GS}.


\begin{Ack}
The third author gratefully acknowledges support from Vergstiftelsen.
\end{Ack}


 
 \section{Preliminaries}

We will consider one-dimensional complex analytic families in $\La_{d,\oli{p'}} \subset \RR_d$ and prove the corresponding main result in almost all such families (where ``almost all'' means almost all tangents in $\La_{d,\oli{p'}}$ in the sense of Levin, see Section \ref{phase-parameter}). The main result will then follow by Fubini's theorem.  Throughout the paper $\QQ = \QQ(\varepsilon)$ will denote a fixed one-dimensional parameter square with sidelength $\varepsilon$, centred around a slowly recurrent Collet--Eckmann map $f_0$. So $f_a$, $a \in \QQ$ is a one-dimensional analytic family of rational maps in $\La_{d, \oli{p'}}$. Let $C_0$ and $\gamma_0$ be the associated constant and exponent of $f_0$ appearing in Definition~\ref{CE} for the {\em starting map} $f_0$. We denote by $c_l(0)$ and $v_l(0)$ (or simply $c_l$ and $v_l$ if it is clear from the context) the critical points and critical values of $f_0$ in $\JJ(f_0)$. In other words, $c_l\in\jrit(f_0)$. The corresponding critical points and critical values for $f_a$ with $a\in\QQ$ are denoted by $c_l(a)$ and $v_l(a)$, respectively, and we see that $c_l(a)$ (and consequently $v_l(a)$) are analytic in $\QQ$ (since we are considering $\La_{d,\oli{p'}}$). For simplicity, we write $\jrit_0$ instead of $\jrit(f_0)$, while $\jrit_a$ denotes the set of $c_l(a)$, for $a\in\QQ$. (Note that we are not claiming that $c_l(a)$ lies in the Julia set $\JJ(f_a)$.) For a connected set $A\subset\QQ$, we let $\jrit_A$ denote the union of $\jrit_a$ over $a \in A$.

For $a\in \QQ$, we are going to study the evolution of the critical points $c_l(a)$, and for this we introduce the functions
\[
\xi_{n,l}(a)=f_{a}^{n}(c_l(a)) \fa n \geq 0.
\]


With $x \lesssim y$ (or $x \gtrsim y$) we will mean that there exists a constant $C > 0$ (not dependent on the dynamics) such that $x \leq C y$ (or $x \geq C y$). If both $x \lesssim y$ and $x \gtrsim y$ then we will write $x \sim y$.

Since $f_0$ satisfies the Collet--Eckmann condition, nearby parameters inherit expansion for some time, and therefore the image of the parameter square $\QQ$ will expand under $\xi_{n,l}$. Once the image of $\QQ$ gets close to $\jrit_{\QQ}$, the derivative will decrease (depending on the distance to $\jrit_{\QQ}$). To ensure that we still have good expansion after getting close to $\jrit_{\QQ}$, a local analysis is needed. Let $0< \Delta' < \Delta$ be two large numbers, and let 
\begin{align*}
\delta=e^{-\Delta}, &\quad \delta' = e^{-\Delta'},\\
U_l = D(c_l,\delta), &\quad U_l' = D(c_l,\delta'),
\end{align*}
and define 
\begin{equation}\label{critn}
U=\bigcup_{l} U_l \quad \text{and} \quad U' = \bigcup_l U_l'
\end{equation}
to be neighbourhoods of the critical points of $f_0$ belonging to $\JJ_0$. By continuity, one can choose $\varepsilon$ sufficiently small such that $U$ is also a neighbourhood of $c_l(a)$ for all parameters $a\in \QQ$. In fact, we want to have $\diam(U_l) \gg \diam(c_l(\QQ))$. We will frequently use local Taylor expansion
\[
f_a(z) = f_a(c_l) + B(z-c_l)^{d_l} + \OO\left((z-c_l)^{d_l+1}\right),
\] 
and $U'$ is chosen to be some fixed neighbourhood where first order Taylor expansions are sufficiently good around any $c_l$. Considering the multiplicity at critical points, we let $\hat{d} = \max_l d_l$. (Note that we assume that the critical points do not split under perturbation, i.e. $d_l = d_l(a)$ is constant for $a \in \QQ$.) The smaller neighbourhood $U$ should be thought of as a neighbourhood that could be as small as one likes to fit into the construction. Furthermore, in the section on large deviations, we will also make use of a smaller neighbourhood $U^2 = \cup_l U_l^2$, where $U_l^2 = D(c_l,\delta^2) \subset U_l$. By choosing $\varepsilon$ small enough we make sure that $U_l^2$ is still a neighbourhood of $\jrit_Q$.

As time evolves, we will discard parameters that come too close to $\jrit_a$. For this reason we define the \emph{basic approach rate assumption} (or simply the \emph{ basic assumption}) as follows.

\begin{Def}
Let $\alpha>0$ and $K>0$ be the constants from the slow recurrence condition (Definition~\ref{SL}). We say that $c_l(a)$ satisfies the basic assumption up to time $n$ with exponent $\alpha$, if
\begin{equation}\label{BA}
\dist(\xi_{k,l}(a), \jrit_a)\geq Ke^{-2\alpha k}~\text{for all~}~k\leq n. 
\end{equation}
\end{Def}

For our starting map $f_0$ which is assumed to be slowly recurrent, the basic assumption is, per definition, always satisfied for all $n$ and for all $l$ with exponent $\alpha$. By making the perturbations sufficiently small, i.e.,  choosing $\varepsilon$ small enough, each parameter $a\in\QQ$ will also satisfy the basic assumption up to some time. However, as the number $n$ of iterates grows, $\xi_{n,l}(\QQ)$ becomes a comparatively large set, so that we shall need to partition our parameter square $\QQ$ in the following way. Let $U$ be as defined in \eqref{critn}.

\begin{Def}[Partition element]\label{partition}
Let $S>0$ be given. A connected set $A\subset \QQ$ is called a partition element at time $n$ if the following holds for all $k\leq n$:
\begin{equation}
\diam\xi_{k,l}(A)\leq
\begin{cases}
~\,~\dfrac{\dist(\xi_{k,l}(A), \jrit_{\QQ})}{(\log \dist(\xi_{k,l}(A), \jrit_{\QQ}))^2} ~&\text{if}~\xi_{k,l}(A)\cap U\neq \emptyset,\\
\quad~\,~S ~&\text{if}~\xi_{k,l}(A)\cap U= \emptyset.
\end{cases}
\end{equation}
\end{Def}

For convenience, the partition elements are going to be squares in our situation, since we start with a square $\QQ$, but in principle this is not needed. 
The reason to make partitions according to the above rule is that we have distortion control of $\xi_{n,l}(a)$ for $a\in A$. So as time evolves, the partition gets finer. The constant $S$ appearing in the above definition is usually referred to as the \emph{large scale}, and we say that a partition element has \emph{escaped} when it reaches size $S$ under the action of the function $\xi_{n,l}$. 

Our main task in the paper is to show that almost all partition elements will reach the large scale within a bounded (but not necessarily uniform) amount of time. This is relatively easy if $\xi_{n,l}(A)$ never comes close to critical points (which is true if our starting map is, say, of Misiurewicz type). In our case, however, it can happen that $\xi_{n,l}(A)$ approaches critical points, since we are starting with a (slowly) recurrent map. Although the approach rate is controlled by the basic assumption, we may still lose derivative. To restore this loss, we shall use the ideas from \cite{BC1, BC2} (see also \cite{MA7} which is similar to the setting here).

The fundamental concepts for dealing with the above situations are the so-called \emph{bound periods} and \emph{free periods}. To define them, we first introduce the notion of \emph{returns} which can be defined for single parameters and also for partition elements. 

Recall that $U$ and $U'$ are defined in \eqref{critn}. For a partition element $A$  we say that $\xi_{n,l}(A)$ is a \emph{return} if $\xi_{n,l}(A)\cap U'\neq\emptyset$ or $\xi_{n,l}(A)\cap U\neq\emptyset$. We speak of a {\em pseudo-returns} if $\xi_{n,l}(A)$ is a return into $U'$ but not $U$.  For a parameter $a\in\QQ$, we say that $\xi_{n,l}(a)$ is a return if $\xi_{n,l}(a)\in U'$ or $\xi_{n,l}(a)\in U$.

\begin{Def}[Bound period for parameters]\label{pointbp}
Let $\alpha$ be as in the basic assumption \eqref{BA}. Let $\xi_{n,l}(a)\in U'_k$ be a return, where $U'_k$ is the component of $U'$ containing $c_k(0)$. The bound period for this return is defined as the indices $j>0$ such that the following holds:
\begin{equation}
\left\vert \xi_{n+j,l}(a)-\xi_{j,k}(a) \right\vert \leq e^{-\alpha j}\dist(\xi_{j,k}(a), \jrit_a).
\end{equation}
The largest number $p = p(a) >0$ for which the above inequality holds is called the length of the bound period.
\end{Def}

During the bound period, the growth of derivative is inherited from its early orbit, regardless of whether or not there are more returns during this period. Such returns are called bound returns. Because of the binding condition in the above definition, these returns are harmless. As soon as the bound period ends, we enter into a \emph{free period}, which means that this piece of orbit stays away from critical points. During the free period, derivative growth is guaranteed by the classical result of Ma\~n\'e (see the next section for a more precise statement). If $p$ is the length of the bound period, when $\xi_{n+p+L,l}(a) \in U'$ for the least possible $L > 0$, we speak of a {\em free return}. The number $L$ is the length of the free period. Since bound returns are harmless we only speak of returns, and thereby mean free returns.  

We will also need a corresponding notion of bound period for partition elements. To define this, let $A \subset \QQ$ be a partition element at time $n$. We say that a return $\xi_{\nu,l}(A)$, $\nu > n$, into $U$ is \emph{essential} if
\begin{equation}
\diam\xi_{\nu,l}(A) \geq \frac{1}{3} \dfrac{\dist(\xi_{\nu,l}(A), \jrit_{\QQ})}{(\log \dist(\xi_{\nu,l}(A), \jrit_{\QQ}))^2}.
\end{equation}
Otherwise, it is called an \emph{inessential return}. When an essential return occurs we will have to make partitions according to Definition~\ref{partition}. Because of strong bounds on distortion, we will see that $\xi_{n,l}$ is almost affine on each partition element $A$ and if $A$ is a perfect square then $\xi_{n,l}(A)$ is also almost a perfect square. If $A$ has side length $d$, simply partition $A$ into four subsquares of equal length. If all these four subsquares are partition elements according to Definition~\ref{partition}, we are done; the new partition is thereby defined. If a subsquare is not a partition element,  continue partitioning it into four new subsquares of equal length, and continue like this until all the new subsquares are partition elements. We get a collection of squares of sidelength of the form $2^{-k}d$, for some $k \geq 0$ (note that we can have different values of $k$). No partition is made at an inessential return. We can now define the bound period for partition elements.

\begin{Def}[Bound period for partition elements]\label{ptbp}
Let $A$ be a partition element at time $n$ and $\xi_{n,l}(A)$ an essential, inessential or pseudo- return to $U'_k$, the component of $U'$. The bound period for this return is defined as indices $j>0$ such that the following holds for all $a,b\in A$ and for all $z\in\xi_{n,l}(A)$:
\begin{equation}
\left\vert f_{a}^{j}(z)- \xi_{j,k}(b)\right\vert \leq e^{-\alpha j}\dist (\xi_{j,k}(b), \jrit_{b}).
\end{equation}
The largest number $p = p(A) >0$ for which the above inequality holds is called the length of the bound period.
\end{Def}

With the above notions, we will follow the parameter exclusion technique originated by M.~Benedicks and L.~Carleson \cite{BC1,BC2}. However, we have to deal with the situation caused by the presence of several critical points (again following an idea due to M. Benedicks). What can happen is that a critical orbit might get close to a critical point other than itself. In this case, to use induction we need to use the binding information of this latter critical point. To handle this we make the following definition. Let $\gamma_I>0$ be a constant to be defined later. 

\begin{Def}
Given $\gamma>0$, we say that a parameter $a$ belongs to $\EE_{n,l}(\gamma)$ if 
\begin{equation}\label{EE1}
\vert Df_{a}^{k}(v_l(a))\vert \geq C_0e^{\gamma k} \fa k\leq n-1,
\end{equation}
and
\begin{equation}\label{EE2}
\vert Df_{a}^{k}(v_j(a))\vert \geq C_0 e^{\gamma k} \fa k\leq 2\hat{d}\alpha n/\gamma_I,~\text{and all}~\,j\neq l.
\end{equation}
We say that a parameter $a$ belongs to $\BB_{n,l}$ if
\begin{equation}\label{BB1}
\dist(\xi_{k,l}(a), \jrit_a)\geq Ke^{-2\alpha k} \fa k\leq n-1,
\end{equation}
and
\begin{equation}\label{BB2}
\dist(\xi_{k,j}(a), \jrit_a)\geq Ke^{-2\alpha k} \fa k\leq 2\hat{d}\alpha n/\gamma_I,~\text{and all}~\,j\neq l.
\end{equation}
\end{Def}

\section{Lemmas}
In this section we present several lemmas on distortion and transversality. The transversality property says that phase and parameter derivatives can be compared if the phase derivative grows at a certain rate. In our new situation with recurrent critical points, this property is inherited from quite recent powerful results by G. Levin \cite{Levin-transv, Levin-book}. Together with a strong distortion lemma in the phase space (the main distortion lemma), we get strong control of the geometry of $\xi_{n,l}(A)$ on partition elements. 

 \subsection{Phase-parameter relations} \label{phase-parameter}
 
 \subsubsection{Transversality}
Using a result by G.~Levin we state a transversality result for Collet--Eckmann parameters, relating phase and parameter derivatives. In the following there is a notion of degenerate families of rational maps, following \cite{Levin-book, Levin-transv}. We consider one-dimensional complex families of rational maps in $\La_{d,\oli{p'}}$ through the starting map $f_0$ such that this family has a non-zero tangent at $f_0$, i.e. such that $f_a(z) = f_0(z) + au(z) + \OO(a^2)$, for some non-zero $u(z)$. For almost all directions of this tangent in the parameter space, it is shown that we have a certain transversality property (see \cite{Levin-transv}, Corollary 2.1, part (8)), namely that the limit
 \[
\lim\limits_{n \raw \infty} \frac{\xi_{n,l}'(0)}{(f_{0}^{n-1})'(f_{0}(c_l))} = L_l,
   \]
   exists and is different from $0$ and $\infty$. (With $\xi_{n,l}^\prime(a)$ we mean the parameter derivative of $f_a^n$ evaluated at $c_l(a)$, i.e. $\xi_{n,l}^\prime(a) = \partial_a f_a^n(c_l(a))$.) Families satisfying this condition are called {\em non-degenerate in the sense of Levin}. Based on this we get the following, see Proposition~4.1 in \cite{MA7}. 
 \begin{Lem} \label{levin}
 Let $f=f_0$ be a slowly recurrent Collet--Eckmann map with exponent $\gamma_0$ and $f_a$, $a \in \QQ$, an analytic non-degenerate family in the sense of Levin. Then for any $q\in (0,1)$ and any $\gamma \in (0,\gamma_0)$ there exists $N >0$ and $\varepsilon >0$ such that
 \begin{equation} 
 \left\vert \frac{\xi'_{n,l}(a)}{(f^{n-1}_{a})'(v_l(a))}-L_l \right \vert \leq q \vert L_l \vert
 \end{equation}
 provided that $f_a$ satisfies the Collet--Eckmann condition up to time $n\geq N$ with exponent $\gamma$ for all $a\in \QQ$. 
 \end{Lem}

Recall that our starting map $f_0$ satisfies the Collet--Eckmann condition with exponent $\gamma_0$.  With $\gamma_H$ being the exponent from Lemma~\ref{oel} below (see Remark~\ref{he}), we shall apply the above lemma for
\[
\gamma_L:=\frac{1}{6}\min\{\gamma_0, \gamma_H\}(1-\tau),
\]
where $\tau\in (0,1)$ is some constant to be determined later. This choice of $\gamma_L$ also dictates the choices of the corresponding $N$ and $\varepsilon$ in Lemma \ref{levin}, which we denote by $N_L$ and $\varepsilon_L$ correspondingly. We choose the size $\varepsilon$ of our domain of perturbation (i.e., the parameter square $\QQ$) to comply with Lemma~\ref{levin}, e.g. $\varepsilon < \varepsilon_L$. For later convenience, we also define 
\[\gamma_{I}:=2\gamma_L,\,~\, ~\text{~and~}~\gamma_B:=\frac{9}{2}\gamma_L.
\]
We thus have that $\gamma_B>\gamma_I>\gamma_L.$
 
\subsubsection{Weak parameter dependence}
The following lemma tells us that the dependence on parameter is weak as long as we have exponential growth of the derivative. As a matter of fact, the dependence is even weaker, as will be seen after the proof of the main distortion lemma.
 
\begin{Lem}\label{weak_parameter_lemma}
Let $N_L$ and $\gamma_L$ be as in Lemma~\ref{levin}, and let $\gamma_1 > (3/2)\gamma_L$. Suppose that $a,b \in \mathcal{Q}$. If $\varepsilon$ and $\delta$ are small enough, and if there is an integer $k_1 \geq 0$ such that
\begin{enumerate}
\item[i)] $\vert Df^n_a(v_l(a))\vert \geq C_1 e^{\gamma_1 n}$ for all $n \leq N_L + k_1$;
\item[ii)] for all $n \leq N_L + k_1$, if $\xi_{n,l}(a),\xi_{n,l}(b) \notin U$ then $\vert \xi_{n,l}(a) - \xi_{n,l}(b) \vert \leq S$, and otherwise if $\xi_{n,l}(a) \in U$ or $\xi_{n,l}(b) \in U$ then
\[
\vert \xi_{n,l}(a) - \xi_{n,l}(b) \vert \leq \frac{\dist(\xi_{n,l}(a'), \jrit_{a'})}{\left(\log\left(\dist(\xi_{n,l}(a'), \jrit_{a'})\right)\right)^2},
\]
with $a' \in \{a,b\}$ minimising $\dist(\xi_{n,l}(a'),\jrit_{a'})$;

\end{enumerate}
then there exists $Q > 1$ (arbitrarily close to $1$ if $N_L$ is large enough) such that for all $N_L \leq n \leq N_L+k_1$
\[
\vert \xi_n(a) - \xi_n(b) \vert \geq Q^{-(n-1)} \vert Df_a^{n-1}(v_l(a))\vert \vert a - b\vert.
\]
Moreover, for all $0 \leq j \leq n-N_L$
\[
\vert \xi_{n,l}(a) - \xi_{n,l}(b) \vert \sim_{Q^j} \vert Df_a^j(\xi_{{n-j},l}(a))\vert \vert \xi_{n-j,l}(a) - \xi_{n-j,l}(b)\vert.
\]
\end{Lem}
\begin{proof}
We fix $l$ and write $\xi_n = \xi_{n,l}$, $N = N_L$. We begin with proving that there is a $Q > 1$ close to $1$ such that
\begin{equation}\label{endpoint_stretch}
\vert \xi_n(a) - \xi_n(b) \vert \geq Q^{-(n-1)} \vert Df_a^{n-1}(v_l(a)) \vert \vert a - b\vert 
\end{equation}
is true for all $N \leq n \leq N+k_1$.

By making $\varepsilon$ small enough we can make sure that
\[
\vert \xi_N(a) - \xi_N(b) \vert \geq \frac{1}{2}\vert \xi_N'(a) \vert \vert a - b\vert.
\]
Using Lemma~\ref{levin} we find that 
\begin{align*}
\vert \xi_N(a) - \xi_N(b) \vert &\geq \frac{1}{2}\vert L_l\vert (1-q) \vert Df_a^{N-1}(v_l(a)) \vert \vert a - b \vert \\
&\geq Q^{-(N-1)} \vert Df_a^{N-1}(v_l(a)) \vert \vert a - b\vert,
\end{align*}
where $Q > 1$ can be made arbitrarily close to $1$ by increasing $N$. Assume that the above inequality holds for some $N \leq n \leq N+k_1 - 1$. We have that
\begin{align*}
\vert \xi_{n+1}(a) - \xi_{n+1}(b) \vert &\geq \vert f_a(\xi_n(a)) - f_a(\xi_n(b)) \vert - \vert f_a(\xi_n(b)) - f_b(\xi_n(b)) \vert \\
&\geq Q_0^{-1} \vert Df_a(\xi_n(a)) \vert \vert \xi_n(a) - \xi_n(b)\vert - 2\vert \partial_a f_a(\xi_n(a)) \vert \vert a - b\vert \\
&\geq Q_0^{-1} Q^{-(n-1)} \left( \vert Df_a^n(v_l(a)) \vert - 2BQ_0 Q^{n-1} \right) \vert a - b\vert,
\end{align*}
where $B = \sup \vert \partial_a f_a \vert$ and $Q_0 > 1$ can be made arbitrarily close to $1$ by making $\epsilon_1$ small enough in $S = \delta \epsilon_1$, and $N$ large enough. 

From assumption i) since if $Q$ is such that $\log Q < \gamma_1/2$, say, then
\begin{equation}\label{pineq}
2BQ_0 Q^{n-1} \leq \frac{Q_0-1}{Q_0} C_1 e^{\gamma_1 n} \leq  \frac{Q_0-1}{Q_0} \vert Df_a^n(v_l(a)) \vert,
\end{equation}
for $N$ large enough. Combining this with with the above, taking $Q_0 = \sqrt Q$, we find that
\[
\vert \xi_{n+1}(a) - \xi_{n+1}(b) \vert \geq Q^{-n} \vert Df_a^n(v_l(a)) \vert \vert a - b \vert,
\]
proving the first conclusion of the lemma.

The proof of the second claim of the lemma is very similar to the proof of the first claim above. We use an inductive argument as follows. For $n = N$ (and thus $j=0$) the result is trivial. Suppose therefore that for some $N \leq n \leq N + k_1 - 1$ the conclusion in the statement of the lemma is true, and consider the case $n+1$. Pick some $0 \leq j \leq n - N$.  
%
%
Using \eqref{endpoint_stretch} we find that
\begin{align*}
\vert &\xi_{n+1}(a) - \xi_{n+1}(b) \vert \\ 
&\geq \vert f_a(\xi_{n}(a)) - f_a(\xi_{n}(b)) \vert - \vert f_a(\xi_{n}(b)) - f_b(\xi_{n}(b)) \vert \\
      &\geq Q_0^{-1} \vert Df_a(\xi_{n}(a)) \vert \vert \xi_{n}(a) - \xi_{n}(b) \vert -
        2\vert \partial_a f_a(\xi_{n}(a)) \vert \vert a - b \vert \\
      &\geq Q_0^{-1} \left(\vert Df_a(\xi_{n}(a)) \vert - \frac{2BQ_0 Q^{n-1}}{\vert Df_a^{n-1}(v_l(a)) \vert} \right) \vert \xi_{n}(a) - \xi_{n}(b) \vert.
\end{align*}
It follows from inequality \eqref{pineq} that
\[
\frac{2BQ_0 Q^{n-1}}{\vert Df_a^{n-1}(v_l(a)) \vert} \leq \frac{Q_0-1}{Q_0} \vert Df_a(\xi_{n}(a)) \vert.
\]
We continue now, using the induction assumption that the lemma is true for $n$, to conclude that, for $0 \leq j \leq n-N$, 
\begin{align}
\vert \xi_{n+1}(a) - \xi_{n+1}(b) \vert &\geq Q_0^{-2} \vert Df_a(\xi_{n}(a)) \vert \vert \xi_{n}(a) - \xi_{n}(b) \vert \\
             &\geq Q_0^{-2} \vert Df_a(\xi_{n}(a)) \vert Q^{-j} \vert Df_a^j(\xi_{n-j}(a)) \vert \vert \xi_{n-j}(a) - \xi_{n-j}(b) \vert \\
             &\geq Q_0^{-2} Q^{-j} \vert Df_a^{j+1}(\xi_{n-j}(a)) \vert \vert \xi_{n-j}(a) - \xi_{n-j}(b) \vert.
\end{align}
Choosing $Q_0 = \sqrt Q$ close enough to $1$, we get
\begin{align*}
\vert &\xi_{n+1}(a) - \xi_{n+1}(b) \vert \geq Q^{-(j+1)} \vert Df_a^{j+1}(\xi_{n-j}(a)) \vert \vert \xi_{n-j}(a) - \xi_{n-j}(b) \vert.
\end{align*}
The case $j=0$ in the second claim of the lemma is trivial. Hence, this proves one of the inequalities of the second claim. We can achieve the other inequality in a completely analogous way.
\end{proof}
With the above lemma we immediately get the following result, telling us that the analytic dependence of critical points on the parameters are negligible.

\begin{Lem}
Under the assumptions of Lemma~\ref{weak_parameter_lemma} there exists a constant $C > 0$ such that
\[
\vert \xi_{n,l}(a) - \xi_{n,l}(b) \vert \geq C e^{\gamma_2 n} \vert c_l(a) - c_l(b) \vert,
\]
for any $N_L \leq n \leq N_L + k_1$.
\end{Lem}
\begin{proof}
Since we assume that the critical points $c_l(a)$ move analytically in $a$ we have 
\[
c_l(a) = c_l(b) + K_l (a - b)^{k_l} + \OO\left( ( a - b)^{k_l+1} \right).
\]
From \eqref{endpoint_stretch} and the the conclusion of Lemma~\ref{weak_parameter_lemma} we find that
\begin{align*}
\vert c_l(a) - c_l(b) \vert &\leq 2\vert K_l \vert k_l \vert a - b \vert^{k_l-1} \vert a - b\vert \\
&\lesssim \frac{Q^n}{\vert Df_a^n(v_l(a)) \vert} \vert \xi_n(a) - \xi_n(b) \vert \\
&\lesssim e^{-\gamma_2 n} \vert \xi_n(a) - \xi_n(b) \vert,
\end{align*}
where $\gamma_2$ is slightly smaller than $\gamma_1$.
\end{proof}
With these lemmas we can neglect the parameter dependence on each partition element $A$. In particular, for returns into $U'$, $\dist(\xi_{n,l}(A),c_l(A))$ is very close to $\vert \xi_{n,l}(a) - c_l(a)\vert$ for all $a \in A$, so we can almost view $c_l(A)$ as one single critical point.

\subsubsection{Distortion estimate} In the sequel we will frequently use the following distortion estimate, which we for convenience formulate as a lemma.
\begin{Lem}\label{dist_esti_0}
If $z$ and $w$ stay sufficiently close to each other under iteration by $f_a$ up to time $n$, then
\begin{equation}\label{dist_esti_1}
\left\vert \frac{Df_a^n(z)}{Df_a^n(w)} - 1 \right\vert \leq \exp \left(C \sum_{j=0}^{n-1} \frac{\vert f_a^j(z) - f_a^j(w) \vert}{\dist(f_a^j(w), \jrit_a)} \right) -1,
\end{equation}
for some constant $C > 0$ dependent on $f_0$ and $\varepsilon$.
Moreover, if also $z = \xi_{\nu}(a)$ and $w = \xi_{\nu}(b)$, with $a,b \in \mathcal{Q}$, and if the assumptions of Lemma~\ref{weak_parameter_lemma} are satisfied, then 
\begin{equation}\label{dist_esti_2}
\left\vert \frac{Df_a^n(\xi_{\nu}(a))}{Df_b^n(\xi_{\nu}(b))} - 1 \right\vert \leq \exp \left(C \sum_{j=0}^{n-1} \frac{\vert \xi_{\nu + j}(a) - \xi_{\nu+j}(b) \vert}{\dist(\xi_{\nu+j}(b), \jrit_b)} \right) -1.
\end{equation}
\end{Lem}
\begin{proof}
Given any complex numbers $z_1,\dots, z_n$, the following inequality is standard
\[
\left\vert \prod_{j = 1}^n z_j - 1 \right\vert \leq \exp\left(\sum_{j=1}^n \vert z_j - 1\vert\right) - 1,
\]
and therefore, using the chain rule, we conclude that
\[
\left\vert \frac{Df_a^n(z)}{Df_b^n(w)} - 1 \right\vert \leq \exp\left( \sum_{j=0}^{n-1} \frac{\vert Df_a(f_a^j(z)) - Df_b(f_b^j(w)) \vert}{\vert Df_b(f_b^j(w))\vert} \right) - 1.
\]
We begin to prove \eqref{dist_esti_1}; \eqref{dist_esti_2} will then follow from Lemma~\ref{weak_parameter_lemma}.

Let us write $z_j = f_a^j(z)$ and $w_j = f_a^j(w)$. If one (or both) of $z_j$ and $w_j$ does not belong to $U'$, we are away from critical points and the estimate follows easily. Indeed we have that
\[
\vert Df_a(z_j) - Df_a(w_j) \vert \leq 2 \sup_{z \in \widehat{\mathbb{C}}} \vert D^2 f_a(z) \vert \vert z_j - w_j\vert,
\]
and also
\[
\vert Df_a(w_j) \vert \geq \frac{1}{2} \frac{\inf_{z \notin U'} \vert Df_a(z) \vert}{\sup_{z \notin U'} \dist(z,\jrit_a)} \dist(w_j,\jrit_a).
\]
If both $z_j$ and $w_j$ belong to $U'_l \subset U'$, using local Taylor expansion
\[
f_a(x) = f(c_l) + B(x-c_l)^{d_l} + \OO\left((x-c_l)^{d_l+1}\right),
\]
and the fact that $\vert Df(z_j) - Df(w_j) \vert \leq \vert D^2f(x) \vert \vert z_j - w_j\vert$ for some $x \in [z_j,w_j]$ we find that
\begin{align*}
\vert Df_a(z_j) - Df_a(w_j) \vert &\leq d_l^2 \vert B\vert \vert x - c\vert^{d_l-2} \vert 1 + \OO\left((x-c)\right) \vert \vert z_j - w_j \vert \\
&\leq 2 d_l^2 \vert B \vert \vert w - c\vert^{d_l-2} \vert z_j - w_j \vert,
\end{align*}
where we used that $\vert z_j - w_j \vert$ and $\delta'$ is very small.

For the derivative we have the estimate
\[
\vert Df_a(w_j) \vert = \vert Bd_l(w_j-c_l)^{d_l-1} + \OO\left((w_j-c_l)^{d_l}\right) \vert \geq \frac{1}{2}d_l \vert B\vert \vert w_j-c \vert^{d_l-1}.
\]
We conclude that
\[
\sum_{j=0}^{n-1} \frac{\vert Df_a(f_a^j(z)) - Df_a(f_a^j(w)) \vert}{\vert Df_a(w) \vert} \lesssim \sum_{j=0}^{n-1} \frac{\vert f_a^j(z) - f_a^j(w) \vert}{\dist(w,\jrit_a)}.
\]

To prove \eqref{dist_esti_2} we use the previous discussion together with Lemma~\ref{weak_parameter_lemma} to conclude that
\begin{align*}
\vert &Df_a(\xi_{\nu+j}(a)) - Df_b(\xi_{\nu+j}(b)) \vert \\
&\leq \vert Df_a(\xi_{\nu+j}(a)) - Df_a(\xi_{\nu+j}(b)) \vert + \vert Df_a(\xi_{\nu+j}(b)) - Df_b(\xi_{\nu+j}(b)) \vert \\
&\lesssim \vert \xi_{\nu+j}(a) - \xi_{\nu+j}(b) \vert + 2 \vert \partial_a Df_a(\xi_{\nu+j}(a))\vert \vert a - b\vert \\
&\lesssim \left(1+  \frac{2 \vert \partial_a Df_a(\xi_{\nu+j}(a))\vert Q^{\nu+j-1}}{\vert Df_a^{\nu+j-1}(v(a))\vert}  \right) \vert \xi_{\nu+j}(a) - \xi_{\nu+j}(b) \vert \\
&\lesssim \vert \xi_{\nu+j}(a) - \xi_{\nu+j}(b) \vert.
\end{align*}
\end{proof}

 \subsection{Expansion during bound periods}

We study in this section how an orbit preserves certain expansion during bound periods (although some loss is unavoidable at returns). To achieve this, we first need some distortion control during these periods. Recall that $U'$ is defined in \eqref{critn}.

\begin{Lem}\label{disbound}
Let $\varepsilon'>0$. Let $\delta'>0$ be sufficiently small and $N$ sufficiently large. Let also $\gamma\geq \gamma_I$. Suppose that $\xi_{\nu,l}(a)$ is a free return to $U'_{i}$ with $a\in\EE_{\nu,l}(\gamma)\cap \BB_{\nu,l}$. Then we have
\begin{equation}
\left\vert \frac{Df_{a}^{j}(\xi_{\nu+1,l}(a))}{Df_{a}^{j}(\xi_{1,i}(a))} - 1 \right\vert \leq\varepsilon'
\end{equation}
for all $j\leq p$, where $p$ is the length of the bound period.
\end{Lem}

\begin{proof}
Since $\xi_{\nu,l}(a)$ is a free return to $U'_i$, we can thus assume that for some $r>0$
\[
\vert\xi_{\nu,l}(a)-c_i(a)\vert \sim e^{-r}.
\]
By Lemma~\ref{dist_esti_0}, it suffices to prove that the following sum can be made sufficiently small
\begin{equation}
\sum_{j=1}^{p}\frac{\vert Df_a(\xi_{\nu+j,l}(a))-Df_{a}(\xi_{j,i}(a))\vert }{\vert Df_a(\xi_{j,i}(a))\vert }\leq C_1\sum_{j=1}^{p}\frac{\vert \xi_{\nu+j,l}(a)-\xi_{j,i}(a)\vert}{\dist(\xi_{j,i}(a), \jrit_a)},
\end{equation}
where $C_1>0$ is some constant. Assume that $d_i$ is the local degree of $f_0$ at $c_i(0)$. So we have, for some constant $C_2>0$,
\begin{equation}\label{local}
\vert \xi_{\nu+1,l}(a)-\xi_{1,i}(a)\vert \leq C_2 e^{-d_i r}.
\end{equation}
Put $J=d_i r/10(\Gamma+2\alpha)$, where $\Gamma = \sup\limits_{z \in \hat{\C}, a\in \QQ} \log \vert f_a(z) \vert$. We can divide the above sum into two parts $[1,J]$ and $[J+1,p]$, and estimate them separately.

For the first sum, we have 
\[
\vert \xi_{\nu+j,l}(a)-\xi_{j,i}(a)\vert\leq e^{\Gamma(j-1)}\vert \xi_{\nu+1,l}(a)-\xi_{1,i}(a)\vert.
\]
Therefore, combining with the basic assumption \eqref{BA} and \eqref{local} we have
\begin{align*}
\sum_{j=1}^{J}\frac{\vert \xi_{\nu+j,l}(a)-\xi_{j,i}(a)\vert}{\dist(\xi_{j,i}(a), \jrit_a)}&\leq\sum_{j=1}^{J}\frac{e^{\Gamma(j-1)}\vert \xi_{\nu+1,l}(a)-\xi_{1,i}(a)\vert}{Ke^{-2\alpha j}}\\
&\leq C_3 e^{-9d_ir/10}\leq C_3 e^{-9d_i\Delta'/10}.
\end{align*}
Here $C_3$ depends only on $C_2$ and $K$.

For the second sum we can use Definition \ref{ptbp} directly to see that
\begin{align*}
\sum_{j=J+1}^{p}\frac{\vert \xi_{\nu+j,l}(a)-\xi_{j,i}(a)\vert }{\dist(\xi_{j,i}(a), \jrit_a)}&\leq  \sum_{j=J+1}^{p}e^{-\alpha j}\leq C_4 e^{-\alpha d_i r/(\Gamma+2\alpha)}\leq C_4 e^{-\alpha d_i \Delta'/(\Gamma+2\alpha)}
\end{align*}
for some constant $C_4>0$. As both of the above sums can be made sufficiently small by choosing $\Delta'$ large enough (i.e., $\delta'$ sufficiently small), we reach the conclusion.
\end{proof}

 \begin{Lem}[Expansion and lengths for bound periods]\label{length}
 Let $\gamma\geq \gamma_I$ and $a\in\EE_{\nu,l}(\gamma)\cap \BB_{\nu,l}$, where $\nu \geq N$ ($N$ as in Lemma \ref{levin}). Assume that $\xi_{\nu,l}(a)$ is a return to $U'_i$ whose length of bound period is $p$. Then if $N$ is sufficiently large we have
 \[
 \left \vert Df_{a}^{p+1}(\xi_{\nu,l}(a)) \right \vert \geq e^{\gamma p/(2d_i)},
 \]
 where $d_i$ is the degree of $f$ at $c_i$.
 
 Moreover, if $\dist(\xi_{\nu,l}(a), \jrit_a)\sim_{\sqrt{e}} e^{-r}$, then
 \[
 \frac{d_i r}{2\Gamma}\leq p\leq \frac{2d_i r}{\gamma}.
 \]
 In particular, $p\leq {2\alpha d_i\nu}/{\gamma}$.
 \end{Lem}

 \begin{proof}
 Recall that $\hat{d}$ is the maximal multiplicity of critical points of $f_0$.  First we show that $p\leq 2\hat{d}\alpha\nu/\gamma_I$ so that we can use the expansion along the orbit of $v_i(a)$ up to time $p$. It follows from Lemma~\ref{disbound} that, for $j\leq p+1$,
 \begin{align}\label{bplo}
\left \vert \xi_{\nu+j,l}(a)-\xi_{j,i}(a) \right \vert &\sim \vert Df^{j-1}_{a}(\xi_{\nu+1,l}(a))\vert \vert \xi_{\nu+1,l}(a)-v_i(a) \vert\\
&\sim \vert Df^{j-1}_{a}(v_i(a)) \vert  \vert \xi_{\nu+1,l}(a)-v_i(a)\vert .\\
 \end{align}
The above relation \eqref{bplo},  combined with the definition of bound period (i.e., Definition \ref{pointbp}), gives us for some constant $C>0$
  \begin{align}\label{esbp}
\vert Df^{j-1}_{a}(v_i(a))\vert \vert \xi_{\nu+1,l}(a)-v_i(a)\vert \leq C e^{-\alpha j}\dist(\xi_{j,i}(a), \jrit_a)\leq C' e^{-\alpha j}.
 \end{align}
Now we see that $p\leq 2\hat{d}\alpha\nu/\gamma_I$. Otherwise, we put $j=2\hat{d}\alpha\nu/\gamma_I$ in \eqref{esbp} and use the fact that $a\in  \EE_{\nu,l}(\gamma)$ (cf. \eqref{EE2}) to obtain that
\[
C_0 e^{\gamma (2\hat{d}\alpha\nu/\gamma_I-1)} e^{-d_i r} \leq C' e^{-\alpha 2\hat{d}\alpha\nu/\gamma_I}.
\]
This means that
\[
\frac{2\hat{d}\alpha\nu}{\gamma_I}\leq \frac{2d_i r}{\gamma},
\]
 which is impossible, since $a\in \BB_{\nu,l}$, $d_i\leq \hat{d}$ and $\gamma>\gamma_I$. This also proves that
\[
p\leq \frac{2d_i r}{\gamma}.
\]
Now it follows from Lemma~\ref{disbound} and $a\in \EE_{\nu,l}(\gamma)$ that
\begin{equation}\label{ej}
E_j:= \vert Df_{a}^{j}(\xi_{\nu+1,l}(a))\vert \sim \vert Df_{a}^{j}(v_i(a))\vert \geq C_0 e^{\gamma j}
\end{equation}
 for $j\leq p$.
 
From \eqref{bplo} we also have that, for $j\leq p+1$,
\begin{equation}\label{bplo2}
\left\vert \xi_{\nu+j,l}(a)-\xi_{j,i}(a) \right\vert \sim \vert Df^{j}_{a}(\xi_{\nu,l}(a))\vert \vert \xi_{\nu,l}(a)-c_i(a)\vert.
 \end{equation}
 With $D_j:= \vert Df_{a}^{j}(\xi_{\nu,l}(a)) \vert$ and \eqref{bplo2} we see that
\begin{equation}\label{bplo3}
D_{p+1}e^{-r}\geq e^{-\alpha(p+1)}\dist(\xi_{p+1,i}(a), \jrit_a)\geq K e^{-3\alpha(p+1)},
\end{equation}
where the first inequality follows from the definition of bound periods and the second one holds since $a\in\BB_{\nu,l}$. By the definition of $\Gamma$ we see from \eqref{bplo3} that
\[
e^{\Gamma p}e^{-d_i r}\geq D_{p+1}e^{-r}\geq K e^{-3\alpha(p+1)}.
\]
So we get that
\[
p\geq \frac{d_i r}{2\Gamma}.
\]
It remains to estimate $D_p$. By \eqref{bplo3},
\begin{equation}
e^{-r}\geq K D_{p+1}^{-1}e^{-3\alpha(p+1)}
\end{equation}
and thus
\begin{equation}
e^{-r(d_i -1)}\geq K^{d_i-1}D_{p+1}^{-(d_i-1)}e^{-3\alpha(p+1)(d_i-1)}.
\end{equation}
As $D_{p+1}\sim e^{-r(d_i-1)}E_p$, we have, by \eqref{ej},
\begin{equation}
D_{p+1}\geq K^{d_i-1}D_{p+1}^{-(d_i-1)}e^{-3\alpha(p+1)(d_i-1)}e^{\gamma p},
\end{equation}
which means that
\begin{equation}
D_{p+1}^{d_i}\geq K^{d_i-1}e^{-3\alpha(p+1)(d_i-1)}e^{\gamma p}.
\end{equation}
So we obtain
\begin{equation}
D_{p+1}\geq K^{(d_i-1)/d_i}e^{-3\alpha(p+1)(d_i-1)/d_i}e^{\gamma p/d_i}\geq e^{\gamma p/(2d_i)}.
\end{equation}
\end{proof}

\subsection{Expansion during free periods}

 Roughly speaking, outside expansion means that the derivative of $f^{n}$ grows exponentially if the orbit stays away from a neighbourhood of the critical points. For slowly recurrent Collet--Eckmann maps, this was proved in \cite{MA7} which will also be crucial in our case. We provide the same statement here (with some modifications). 

We begin with stating the following classical lemma by Ma\~n\'e \cite{RM}.
 
 \begin{Lem}\label{mane}
Let $f$ be a rational map. Provided $\delta$ is small enough, there exist a constant $C_M > 0$, dependent on $\delta$, and an exponent $\lambda_M > 1$ such that if $z \in \JJ(f)$ and $f^k(z) \notin U$ for $k = 1,2,\dots n-1$ then
\[
\vert Df^n(z) \vert \geq C_M \lambda_M^n.
\] 
\end{Lem}

If $f^n(z)$ is a return to $U$, one can say something stronger. To state our next lemma we recall the definition of the second Collet--Eckmann condition.
\begin{Def}[Second Collet--Eckmann condition]\label{CE2}
A non-hyperbolic rational map $f$ without parabolic periodic points is said to satisfy the \emph{second Collet--Eckmann condition}, if there exist $C>0$ and $\gamma>0$ such that for every $n \geq 1$ and $w \in f^{-n}(c)$, for $c \in \jrit(f)$ not in the forward orbit of other critical points,
\[
\vert Df^{n}(w) \vert \geq Ce^{\gamma n}.
\]
\end{Def} 
In general, the Collet--Eckmann condition does not imply the second Collet--Eckmann condition, and vice versa. However, within the family of slowly recurrent rational maps, these two conditions are equivalent \cite{Mats-2}. By the assumptions imposed on our starting map $f_0$, we therefore have that it satisfies the second Collet--Eckmann condition. The following lemma ensures strong expansion for orbits outside of $U$. We give an outline of the proof which technically is similar to the proof of Lemma~2.3 in \cite{PRS}. For a detailed proof we refer to Lemma~3.1 in \cite{MA7}.
\begin{Lem}\label{pre-outside}
Let $f$ be a rational map satisfying the second Collet--Eckmann map with exponent $\gamma > 1$, and let $U = \bigcup_l U_l$ be such that $U_l = D(c_l,\delta)$ is a neighbourhood of $c_l \in \jrit(f)$.
If $\delta$ is small enough there exists a constant $C > 0$, not dependent on $\delta$, such that if
\[
z, f(z),\dots,f^{n-1}(z) \notin U,\,~ f^n(z) \in U_k,
\]
with $c_k$ not in the forward orbit of other critical points, then
\[
\vert Df^n(z) \vert \geq C e^{\gamma n}.
\]
 \end{Lem}
\begin{proof}
Let $W_j$ denote the connected component of $f^{-j}(U_k)$ containing $z_j = f^{n-j}(z)$, and let (with some abuse of notation) $c_j$ denote the $j$-th preimage of the critical point $c = c_k$ contained in $W_j$, i.e. $f^j(c_j) = c$ and $c_j \in W_j$. Following the proof of Lemma~2.3 in \cite{PRS}, once a small $\delta_0 > 0$ is fixed there is $\ell \geq 1$ (dependent on $\delta_0$) such that, if $\delta < \delta_0$ is small enough,
\begin{equation}\label{ExpShrink}
\diam(W_j) \leq e^{-\gamma' j} \diam(U_k) = e^{-\gamma' j} \delta,
\end{equation}
for all $j > \ell$. Here $\gamma' > 0$ is slightly smaller than the exponent from the second Collet--Eckmann condition. We now consider the quotient
\[
\left\vert\frac{Df^n(c_n)}{Df^n(z)}\right\vert = \left\vert \frac{Df^{n-\ell}(c_n)}{Df^{n-\ell}(z)} \right\vert \left\vert \frac{Df^{\ell}(c_{\ell})}{Df^{\ell}(z_{\ell})} \right\vert.
\]
By making $\delta$ small, the second factor in the above is bounded by some small constant $C' \geq 1$. For the first factor, using \eqref{dist_esti_1}, \eqref{ExpShrink}, and the assumption that $\dist(z_j,\jrit(f)) \geq \delta$ for $j \geq 1$, we find that
\begin{align*}
\left\vert \frac{Df^{n-l}(c_n)}{Df^{n-l}(z)} \right\vert &\leq \exp\left(C'' \sum_{j=0}^{n-l-1} \frac{\vert f^j(c_n) - f^j(z)\vert}{\dist(f^j(z),\jrit(f))}\right) \\
&\leq \exp\left(C'' \sum_{j = l+1}^n \frac{\diam(W_j)}{\delta}\right) \\
&\leq \exp\left(\frac{C''}{e^{\gamma'}-1}\right).
\end{align*}
This proves the result, since
\[
\vert Df^n(z) \vert \geq \frac{1}{C'}\exp\left(-\frac{C''}{e^{\gamma'}-1}\right) \vert Df^n(c_n) \vert \gtrsim e^{\gamma n}.
\]
\end{proof}

We are now in position to prove our desired outside expansion lemma, satisfied for any small perturbation $f_a$ of $f_0$, and also valid in an $\varepsilon_0$-neighbourhood $\mathcal{N}_{\varepsilon_0}$ of $\JJ(f_0)$. We make $U$ so small that $U \subset \mathcal{N}_{\varepsilon_0}$. 
\begin{Lem}\label{oel}
\text
Suppose $f = f_0$ is a rational second Collet--Eckmann with non-empty Fatou set. If $\varepsilon_0$, $\delta$ and $\varepsilon$ are small enough, there exist constants $C_\delta > 0$ (dependent on $\delta$) and $\gamma > 0$ such that, for all $a \in \mathcal{Q}$, if
\[
z,f_a(z),\dots f_a^{n-1}(z) \in \mathcal{N}_{\varepsilon_0} \setminus U
\]
then
\begin{equation}\label{outside_no_hit}
\vert Df_a^n(z) \vert \geq C_\delta e^{\gamma n}
\end{equation}
Moreover, there exists a constant $C > 0$ (not dependent on $\delta$) such that if we also have $f_a^n(z) \in U$ then
\begin{equation}\label{outside_hit}
\vert Df_a^n(z) \vert \geq C e^{\gamma n}.
\end{equation}
\end{Lem}

\begin{Rem}\label{he}
For later convenience, we put $\gamma_H= \gamma$, where $\gamma$ is as in \eqref{outside_hit}.
\end{Rem}

\begin{proof}
Fix $\gamma > 0$ such that $3 \gamma \in (0,\min\{\gamma_M,\gamma_H\})$, where $\gamma_M$ and $\gamma_H$ comes from Lemma~\ref{mane} and Lemma~\ref{pre-outside}, respectively. We will establish the result with this $\gamma$.

From Lemma~\ref{mane}, provided $\delta$ is small enough, we can find $\hat{n}$ large enough such that if $f^k(z) \in \JJ(f) \setminus U$ for $k = 0,1,\dots, \hat{n}-1$ then
\[
\vert Df^{\hat{n}}(z) \vert \geq C_M e^{\gamma_M\hat{n}} \geq e^{3 \gamma \hat{n}}.
\]
If $\varepsilon_0 > 0$ is small enough, we therefore conclude by continuity that if $f^k(z) \in \mathcal{N}_{\varepsilon_0} \setminus U$ for $k = 0,1,\dots, \hat{n}-1$ then
\[
\vert Df^{\hat{n}}(z) \vert \geq e^{2\gamma\hat{n}}.
\]
Using continuity again, now in the parameter variable, we conclude that if $\varepsilon$ is small enough, if for $a \in \mathcal{Q}$ we have that $f_a^k(z) \in \mathcal{N}_{\varepsilon_0} \setminus U$ for $k = 0,1,\dots,\hat{n}-1$ then
\[
\vert Df_a^{\hat{n}}(z) \vert \geq e^{\gamma \hat{n}}.
\]

Suppose now that $f^k_a(z) \in \mathcal{N}_{\varepsilon_0} \setminus U$ for $k = 0,1,\dots n-1$ and write $n = q \hat{n} + r$, with $q$ and $r$ positive integers and $0 \leq r \leq \hat{n}-1$. We find that
\begin{align*}
\vert Df_a^n(z) \vert &= \vert Df_a^r(f_a^{q\hat{n}}(z)) \vert \vert Df_a^{\hat{n}}(f_a^{(q-1)\hat{n}}(z))\vert \cdots \vert Df_a^{\hat{n}}(z) \vert \\
&\geq \vert Df_a^r(f^{q\hat{n}}(z))\vert e^{\gamma q\hat{n}},
\end{align*}
and \eqref{outside_no_hit} now follows with constant
$C_\delta = \inf_{a \in \mathcal{Q}} \inf_{z \in \mathcal{N}_{\varepsilon_0} \setminus U}\vert Df_a^{\hat{n}} (z)\vert$. 

If we have a return to $U$ then we are in the situation $f_a^k(f_a^{q\hat{n}}) \notin U$ for $k = 0,\dots,r-1$ and $f_a^r(f_a^{q\hat{n}}(z)) \in U$. In the case of the unperturbed map we get from Lemma~\ref{pre-outside} that
\[
\vert Df^r(f^{q\hat{n}}(z)) \vert \geq C e^{\gamma_H r},
\]
with $C$ not depending on $\delta$. Once again, due to continuity in $a$, a similar estimate holds for the perturbed map $f_a$ if $\varepsilon$ is small enough (recall that $r \leq \hat{n}-1$). This proves \eqref{outside_hit}.

\end{proof}

\subsection{At the next free return}

Combining the results above we draw the following conclusions at the next free return.

Using Lemma~\ref{length} and Lemma~\ref{oel}, we can obtain longer time of exponential growth for derivatives. More precisely, we have the following.

\begin{Lem}\label{expgrow}
Let $N$ be sufficiently large. Let $\nu \geq N$ be a return such that $\xi_{\nu,l}(a)\in U'$ for $a\in\EE_{\nu,l}(\gamma)\cap\BB_{\nu,l}$, where $\gamma\geq \gamma_I$. Let also $\nu'$ be the next free return time. Then we have
\[
\vert Df_{a}^{n}(v_{l}(a)) \vert \geq e^{\gamma_1 n},
\]
for all $0 \leq n \leq \nu'-1$, with $\gamma_1\geq (9/10)\min\{\gamma, \gamma_H\}$.
\end{Lem}

\begin{proof}
Recall that $\hat{d}$ is the maximal multiplicity of critical points of $f_0$. Let $p$ be the length of the bound period for the return $\xi_{\nu,l}(a)$, and suppose $n = \nu + j$ with $1 \leq j \leq p$. By the chain rule, the fact that $a\in\EE_{\nu,l}(\gamma)$ and Lemma~\ref{length}, using the notation from the proof of Lemma~\ref{length}, 
\begin{align*}
\vert Df_{a}^{\nu +j}(v_{l}(a)) \vert &= \vert Df_{a}^{\nu -1}(v_l(a)) \vert D_{j+1} \\
&\gtrsim e^{\gamma(\nu -1)} \vert Df_a(\xi_{\nu}(a)) \vert E_j \\
&\gtrsim e^{\gamma(\nu - 1)} e^{-2\alpha \hat{d}\nu} e^{\gamma j} \\
&\geq e^{\gamma_1 (\nu + j)},
\end{align*}
by the choice of $\alpha$ and provided $N$ is large enough. If $n = \nu + p + j$, with $1 \leq j \leq L-1$, it also follows from Lemma~\ref{oel} (here with respect to $U'$) and the above that
\begin{align*}
\vert Df_{a}^{\nu + p + j}(v_{l}(a))\vert &= \vert Df_{a}^{\nu-1}(v_l(a))\vert D_{p+1} \vert Df_a^j(\xi_{\nu+p+1}(a)) \vert \\
&\gtrsim e^{\gamma(\nu-1)} e^{-2\alpha \hat{d} \nu} e^{\gamma p} C_{\delta'} e^{\gamma_H j} \\
&\geq e^{\gamma_1(\nu+p+j)},
\end{align*}
again provided that $N$ is large enough.
\end{proof}

By the weak parameter dependence (Lemma~\ref{weak_parameter_lemma}) and using similar methods as above we can see that parameters belonging to the same partition element repel each other in the following sense.

 \begin{Lem}\label{gro}
 Let $a,b\in\EE_{\nu,l}(\gamma)\cap\BB_{\nu,l}$ be in the same partition element for $\gamma\geq\gamma_I$. Let also $\nu$ be a return time for this partition element. If $\nu'$ is the next free return, then
\[
\vert \xi_{\nu',l}(a)-\xi_{\nu',l}(b) \vert \geq 2 \vert \xi_{\nu,l}(a)-\xi_{\nu,l}(b)\vert.
\]
\end{Lem}
 
\begin{proof}
By Lemma~\ref{expgrow} above we see that for $a\in\EE_{\nu,l}(\gamma)\cap\BB_{\nu,l}$ we have exponential growth of phase derivative up to the next free return $\nu'$, i.e.
\[
\vert Df^{n}_{a}(v_l(a)) \vert \geq C_1 e^{\gamma_1 n},
\] 
for $0 \leq n \leq \nu'-1$. Since $\gamma_1\geq (9/10)\min\{\gamma, \gamma_H\}$, one can use the weak parameter dependence property to get
\[
\vert \xi_{\nu',l}(a)-\xi_{\nu',l}(b)\vert \geq Q^{-(\nu'-\nu)} \vert Df_{a}^{\nu'-\nu}(\xi_{\nu,l}(a))\vert \vert \xi_{\nu,l}(a)-\xi_{\nu,l}(b) \vert.
\]
Since $\nu'-\nu = p+L$, with $p$ and $L$ being the associated bound period and free period, respectively, we get from Lemma~\ref{length} and Lemma~\ref{oel} that
\begin{align*}
\vert Df_a^{\nu'-\nu}(\xi_{\nu,l}(a)) \vert &= \vert Df_a^{p+1}(\xi_{\nu,l}(a)) \vert \vert Df_a^{L-1}(\xi_{\nu + p + 1}(a))\vert \\
&\geq C e^{\gamma p /(2\hat{d}) + \gamma_H(L-1)} \\
&\geq e^{\gamma_2 (\nu' - \nu)},
\end{align*}
for some $\gamma_2 > 0$, provided $\delta'$ is small enough (hence $p$ is large). Notice that the constant $C$ coming from the outside expansion lemma is not dependent on $\delta'$ since we have an actual return.
As $Q$ is chosen very small we have $\log Q<\alpha\ll\gamma_2$, and we get the conclusion.
\end{proof}

\subsection{Main distortion lemma (MDL)}
The following lemma is the main result of this section, and it tells us that we have strong distortion estimates for parameters belonging to the same partition element. An essential ingredient in the proof is the weak parameter dependence proved earlier (Lemma~\ref{weak_parameter_lemma}).
 
\begin{Lem}[Main distortion lemma]\label{md}
Let $\varepsilon'>0$. Then there exists $N$ large enough such that the following holds: If $A\subset\EE_{\nu,l}(\gamma)\cap\BB_{\nu,l}$ is a partition element for $\gamma\geq\gamma_I$ and $\nu \geq N$ is a return time or does not belong to the bound period, and $\nu'$ is the next free return, then we have
\begin{equation} 
\left\vert \frac{Df_{a}^{n}(v_l(a))}{Df_{b}^{n}(v_l(b))} -1\right \vert \leq\varepsilon'
\end{equation}
for $a,b\in A$ and for $\nu\leq n\leq\nu'$, provided $A$ is still a partition element at the time $n$.
\end{Lem}

\begin{proof}
By Lemma~\ref{weak_parameter_lemma} and Lemma~\ref{dist_esti_0}, it reduces to check whether the following sum can be made arbitrarily small:
\begin{equation}
\Upsilon:=\sum_{j=1}^{n-1}\frac{\vert \xi_{j,l}(a)-\xi_{j,l}(b) \vert}{\dist(\xi_{j,l}(b), \jrit_b)}.
\end{equation}
Let $(\nu_k)$ be the free returns before time $n$, where $k\leq s$. In other words, $\nu=\nu_s$ and $\nu'=\nu_{s+1}$. Let also $p_k$ be the length of the associated bound period of the return $\nu_k$. The estimate of $\Upsilon$ is divided into several parts:
\begin{align*}
\Upsilon=\sum_{k=1}^{s}\sum_{j=\nu_{k-1}}^{\nu_{k-1}+p_{k-1}}\frac{\vert \xi_{j,l}(a)-\xi_{j,l}(b)\vert}{\dist(\xi_{j,l}(b), \jrit_b)}&+\sum_{k=1}^{s}\sum_{j=\nu_{k-1}+p_{k-1}+1}^{\nu_{k}-1}\frac{\vert \xi_{j,l}(a)-\xi_{j,l}(b)\vert}{\dist(\xi_{j,l}(b), \jrit_b)}\\
&+\sum_{j=\nu_s}^{n-1}\frac{\vert \xi_{j,l}(a)-\xi_{j,l}(b)\vert }{\dist(\xi_{j,l}(b), \jrit_b)}\\
&=:\Upsilon_B + \Upsilon_F + \Upsilon_T.
\end{align*}
Here $\Upsilon_B$ denotes the contribution from bound periods, while $\Upsilon_F$ the contribution from free periods, and $\Upsilon_T$ the contribution from the last return $\nu_s$ up until time $n$.

\medskip
\noindent{\emph{Contribution from bound periods: the estimate of} $\Upsilon_B$.} Let $\nu_k$ be one of the free returns, with $k\leq s-1$. We would like to estimate the following
\[
\Upsilon_{B}^{k}:=\sum_{j=\nu_k}^{\nu_k+p_k}\frac{\vert \xi_{j,l}(a)-\xi_{j,l}(b) \vert}{\dist(\xi_{j,l}(b), \jrit_b)}=\sum_{j=0}^{p_k}\frac{\vert \xi_{\nu_k+j,l}(a)-\xi_{\nu_k+j,l}(b) \vert}{\dist(\xi_{\nu_k+j,l}(b), \jrit_b)}.
\]
Assume also that $\xi_{\nu_k,l}(a)\in U'_i$ is a return and $\dist(\xi_{\nu_k,l}(b), \jrit_b)\sim e^{-r}$. It then follows from the distortion in Lemma~\ref{disbound} and the definition of bound periods that
\begin{align*}
\Upsilon_{B}^{k}&\leq \frac{\vert \xi_{\nu_k,l}(a) - \xi_{\nu_k,l}(b)\vert}{e^{-r}} \left(1 + \sum_{j=1}^{p_k}\frac{\vert Df_{a}^{j}(\xi_{\nu_k,l}(a))\vert e^{-r}}{\dist(\xi_{\nu_k+j,l}(b), \jrit_b)} \right)\\
&\lesssim \frac{\vert \xi_{\nu_k,l}(a) - \xi_{\nu_k,l}(b)\vert}{e^{-r}}\left(1 + \sum_{j=1}^{p_k}\frac{\vert Df_{a}^{j}(\xi_{\nu_k,l}(a))\vert \vert \xi_{\nu_k, l}(a)-c_i(a)\vert }{\dist(\xi_{\nu_k+j,l}(b), \jrit_b)}\right)\\
&\lesssim \frac{\vert \xi_{\nu_k,l}(a) - \xi_{\nu_k,l}(b)\vert}{e^{-r}}\left(1 + \sum_{j=1}^{p_k}\frac{\vert \xi_{\nu_k+j,l}(a)-\xi_{j,i}(a)\vert}{\dist(\xi_{\nu_k+j,l}(b), \jrit_b)} \right)\\
&\lesssim \frac{\vert \xi_{\nu_k,l}(a) - \xi_{\nu_k,l}(b)\vert}{e^{-r}}\left(1 +  \sum_{j=1}^{p_k} e^{-\alpha j}\right)\\
&\lesssim \frac{\vert \xi_{\nu_k,l}(a) - \xi_{\nu_k,l}(b)\vert}{e^{-r}}.
\end{align*}
Given $r\geq\Delta$, let $K(r)$ be the set of indices $k$ such that $\dist(\xi_{\nu_k,l}(A), \jrit_A) \sim e^{-r}$, and let $\hat{k}(r)$ be the largest index contained in $K(r)$. Then it follows from Lemma~\ref{gro} that
\begin{equation}
\Upsilon_B=\sum_{k=1}^{s-1}\Upsilon_{B}^{k}\leq\sum_{r\geq\Delta}\sum_{k\in K(r)}\Upsilon_{B}^{k}\lesssim \sum_{r\geq\Delta}\Upsilon_{B}^{\hat{k}(r)} \lesssim \sum_{r\geq\Delta} \frac{1}{r^2}\lesssim \frac{1}{\Delta},
\end{equation}
where we used that for the last return associated with this $r$
\[
\vert \xi_{\nu_{\hat{k}(r)},l}(a) - \xi_{\nu_{\hat{k}(r)},l}(b)\vert \lesssim \frac{e^{-r}}{r^2}.
\]

\medskip
\noindent{\emph{Contribution from free periods: the estimate of} $\Upsilon_F$.} Similar as above, we define for $k\leq s-1$,
\[
\Upsilon_{F}^{k}:=\sum_{j=\nu_k+p_k+1}^{\nu_{k+1}-1}\frac{\vert \xi_{j,l}(a)-\xi_{j,l}(b)\vert}{\dist(\xi_{j,l}(b), \jrit_b)}.
\]
By the weak parameter dependence and Lemma~\ref{oel} we see that
\begin{align*}
\vert \xi_{\nu_{k+1},l}(a)-\xi_{\nu_{k+1},l}(b)\vert &\geq \frac{1}{Q^{\nu_{k+1}-j}} \vert Df_{a}^{\nu_{k+1}-j}(\xi_{j,l}(a))\vert \xi_{j,l}(a)-\xi_{j,l}(b)\vert \\
&\gtrsim \left(\frac{e^{\gamma_H}}{Q}\right)^{\nu_{k+1}-j} \vert \xi_{j,l}(a)-\xi_{j,l}(b) \vert
\end{align*}
for $\nu_{k}+p_k+1\leq j\leq \nu_{k+1}-1$. Since $\nu_{k+1}$ is the index of a return, we assume that $\dist(\xi_{\nu_{k+1},l}(b), \jrit_b)\sim e^{-r}.$ So we see that, for $\nu_{k}+p_k+1\leq j\leq \nu_{k+1}-1$,
\begin{align*}
\vert \xi_{j,l}(a)-\xi_{j,l}(b)\vert \lesssim \left(\frac{Q}{e^{\gamma_H}}\right)^{\nu_{k+1}-j}\vert \xi_{\nu_{k+1},l}(a) - \xi_{\nu_{k+1},l}(b)\vert .
\end{align*}
Since $\dist(\xi_{j,l}(b),\jrit_b) \geq \dist(\xi_{\nu_{k+1},l}(b),\jrit_b) \sim e^{-r}$ this gives
\[
\Upsilon_{F}^{k}\lesssim \frac{\vert \xi_{\nu_{k+1},l}(a) - \xi_{\nu_{k+1},l}(b)\vert}{e^{-r}} \sum_{j=\nu_{k}+p_k+1}^{\nu_{k+1}-1}\left(\frac{Q}{e^{\gamma_H}}\right)^{\nu_{k+1}-j}\lesssim \frac{\vert \xi_{\nu_{k+1},l}(a) - \xi_{\nu_{k+1},l}(b)\vert}{e^{-r}},
\]
where we have used the fact that $\log Q$ is much smaller than $\gamma_H$. Using the same argument as in the estimate of the contribution from the bound periods, we find that
\[
\Upsilon_F=\sum_{k=1}^{s-1}\Upsilon_{F}^{k}\leq\sum_{r\geq\Delta}\sum_{k\in K(r)}\Upsilon_{F}^{k}\lesssim \sum_{r\geq\Delta}\Upsilon_{F}^{\hat{k}(r)} \lesssim \sum_{r\geq\Delta} \frac{1}{r^2}\lesssim \frac{1}{\Delta}.
\]

\medskip
\noindent{\emph{Estimate of tail} $\Upsilon_T$.} It remains to estimate the sum between the last free return $\nu_{s}$ up to time $n$. As $\nu_s\leq n\leq \nu_{s+1}$, we need to consider the different situations that can occur in this time interval. If $n\leq \nu_s+p_s$ (i.e. $n$ belongs to the bound period immediately after $\nu_s$), then the tail $\Upsilon_T$ can be estimated in the same as $\Upsilon_B$ by reducing to the return at time $\nu_{s}$. If $n=\nu_{s+1}$, then the tail consists of a bound period following $\nu_s$ and a free period before the return $n=\nu_{s+1}$ happens. In this case, the estimate $\Upsilon_T$ can be estimated again as above.

The remaining case is when $\nu_{s}+p_s+1\leq n< \nu_{s+1}$. For this purpose, we consider pseudo-returns, and we let $\nu_{s}+p_s+1\leq q_1\leq\cdots\leq q_t\leq n$ be the indices of these returns. By definition, $\xi_{q_k,l}(a)\cap U'\neq\emptyset$ and $\xi_{q_k,l}(a)\cap U=\emptyset$. For pseudo-returns, bound periods and free periods are defined in a similar way. As in the previous estimates, the contribution to the distortion between any two pseudo-returns of index $q_k$ and $q_{k+1}$ is a constant times
\[
\frac{\vert \xi_{q_k,l}(a) - \xi_{q_k,l}(b) \vert }{e^{-r_k}} + \frac{\vert \xi_{q_{k+1},l}(a) - \xi_{q_{k+1},l}(b) \vert}{e^{-r_{k+1}}},
\]
where $r_k$ and $r_{k+1}$ such that $\dist(\xi_{q_k,l}(b), \jrit_a) \sim e^{-r_k}$ and $\dist(\xi_{q_{k+1},l}(b), \jrit_a) \sim e^{-r_{k+1}}$. The difference here is that, at a pseudo-return, the only thing we know about the length of our interval is that $\vert \xi_{q_k,l}(a) - \xi_{q_k,l}(b) \vert \leq S$, where $S = \varepsilon_1 \delta$ is the large scale. With similar methods and notation used for estimating the bound and free contributions, we have
\begin{align*}
\Upsilon_T &= \left(\sum_{j = \nu_s}^{q_1} + \sum_{k = 1}^{t-1} \sum_{j = q_k}^{q_{k+1}-1} + \sum_{j = q_t}^{n-1}\right) \frac{ \vert \xi_{j,l}(a) - \xi_{j,l}(b)\vert}{\dist (\xi_{j,l}(b),\jrit_b)} \\
&\lesssim \frac{1}{r_s^2} + \sum_{k = 1}^t \frac{\vert \xi_{q_k,l}(a) - \xi_{q_k,l}(b)\vert}{\dist(\xi_{q_k,l}(b),\jrit_b)} + \sum_{j = q_t}^{n-1} \frac{ \vert \xi_{j,l}(a) - \xi_{j,l}(b)\vert}{\dist (\xi_{j,l}(b),\jrit_b)} \\
&\lesssim \frac{1}{\Delta^2} + \sum_{r = \Delta'}^\Delta \frac{\vert \xi_{q_{\hat{k}(r),l}}(a) -  \xi_{q_{\hat{k}(r),l}}(b) \vert}{\dist(\xi_{q_{\hat{k}(r)},l}(b),\jrit_b)} + \frac{S}{\delta'} \\
&\lesssim \frac{1}{\Delta^2} + \varepsilon_1 \sum_{r = \Delta'}^\Delta e^{r - \Delta} + \varepsilon_1 \\
&\lesssim \frac{1}{\Delta^2} + \varepsilon_1,
\end{align*}
where we in the sum from $q_t$ to $n-1$ used Lemma~\ref{oel} (inequality \eqref{outside_no_hit}, now with respect to $U'$) and that $\dist(\xi_{j,l}(b),\jrit_b) > \delta' > \delta$ during this time.

Combining all these estimate above we arrive at 
\[
\Upsilon = \sum_{j=1}^{n-1} \frac{\vert \xi_{j,l}(a) - \xi_{j,l}(b) \vert}{\dist(\xi_{j,l}(b),\jrit_b)} \lesssim \frac{1}{\Delta} + \varepsilon_1,
\]
and if $\delta$ and $\varepsilon_1$ are small enough, we reach the desired conclusion of strong distortion.
\end{proof}

\subsection{Consequences of MDL}

With Lemma~\ref{md} in hand, we can conclude that the previously obtained weak parameter dependence of Lemma~\ref{weak_parameter_lemma} can be promoted to a stronger form:
\begin{equation}
\vert \xi_{n+j,l}(a)-\xi_{n+j,l}(b) \vert \sim \vert Df^{n}_{a}(\xi_{j,l}(a))\vert \vert \xi_{j,l}(a)-\xi_{j,l}(b)\vert,
\end{equation}
provided that $a$ and $b$ belong to the same partition element.

Another direct consequence of Lemma~\ref{md} is that for a sufficiently small parameter square $\QQ$, we have the following dichotomy for each critical point $c_l$: there exists $N_l$ such that either $\xi_{N_l,l}(\QQ)$ grows to some definite size or $\xi_{N_l,l}(\QQ)$ is the first essential return.

\begin{Lem}\label{startup}
Let $f_0$ be a slowly recurrent Collet--Eckmann rational map. Let $N_L$ be as in Lemma~\ref{levin}, and let $\varepsilon'>0$ be sufficiently small. Then there is a neighbourhood $U$ of $\jrit_0$ and $S>0$ (depending on $U$) such that, for each sufficiently small $\varepsilon>0$ and for each critical point $c_l(0)\in\jrit_0$, there is $N_l\geq N_L$ such that for all $a\in\QQ$ we have the following:
\begin{itemize}
\item[(i)] For some $\gamma_l\geq\gamma_0(1-\varepsilon')$, one has
\[
\left \vert Df_{a}^{k}(v_{l}(a))\right \vert \geq Ce^{\gamma_l k}\quad\text{for}\quad k\leq N_l-1;
\]

\item[(ii)] for $k\leq N_l-1$, one has
\begin{equation}
\diam\xi_{k,l}(\QQ)\leq
\begin{cases}
~\,~\dfrac{\dist(\xi_{k,l}(\QQ), \jrit_{\QQ})}{(\log \dist(\xi_{k,l}(\QQ), \jrit_{\QQ}))^2},~&\text{if}~\xi_{k,l}(\QQ)\cap U\neq \emptyset,\\
\quad~\,~S, ~&\text{if}~\xi_{k,l}(\QQ)\cap U= \emptyset;
\end{cases}
\end{equation}

\item[(iiii)] for $k=N_l$, one has
\begin{equation}
\diam\xi_{k,l}(\QQ)\geq
\begin{cases}
~\,~\dfrac{\dist(\xi_{k,l}(\QQ), \jrit_{\QQ})}{(\log \dist(\xi_{k,l}(\QQ), \jrit_{\QQ}))^2},~&\text{if}~\xi_{k,l}(\QQ)\cap U\neq \emptyset,\\
\quad~\,~S, ~&\text{if}~\xi_{k,l}(\QQ)\cap U= \emptyset;
\end{cases}
\end{equation}

\item[(iv)] for all $a,b\in\QQ$ one has
\[
\left\vert \frac{Df_{a}^{n-N}(\xi_{N,l}(a))}{Df_{a}^{n-N}(\xi_{N,l}(b))}-1 \right \vert\leq\varepsilon' \quad\text{for}\quad n\leq N_l.
\]
\end{itemize}
\end{Lem}
 
 \begin{proof}
By the choice of $N_L$, we can choose $\varepsilon>0$ sufficiently small such $\QQ\subset\EE_{N_L,l}(\gamma)\cap\BB_{N_L,l}$ for all $l$ and for some $\gamma$ arbitrarily close to $\gamma_0$.

Now we fix any $c_l(0)\in\jrit_0$ and assume that $(ii)$ is always satisfied up to some time, denoted by $N_l-1$. Then this implies that all other parameters in $\QQ$ will inherit expansion from our starting map:
\[
\left\vert Df_{a}^{k}(v_l(a)) \right\vert \geq C^{-k}e^{\gamma_0 k}\geq C_0 e^{\gamma_1 k}.
\]
Since we assumed that $(ii)$ is satisfied, we see that all parameters will be slowly recurrent up to time $N_l-1$. Then by the definition of partition element we can use Lemma~\ref{md} repeatedly starting from the time $N_L$ up to $N_l-1$ to get the distortion claimed in $(iv)$.
\end{proof}

To prove our main result, we would like to see if a small parameter square will grow to the large scale $S$ under the action of $\xi_{n,l}$. For this purpose, let $N_l$ be as in the above lemma and suppose without loss of generality that $N_1\leq N_2\leq\cdots$. Then by Lemma~\ref{startup}, we have the situation that $\xi_{N_1,1}(\QQ)$ either reaches the large scale $S$ or is the first essential return. If the first case happens, we stop and consider the next critical point. If the second case occurs, we partition the parameter square $\QQ$ into small dyadic squares inductively as follows. Since the partition rule should be valid for all returns, let us consider a given partition element $A \subset \QQ$ (instead of $\QQ$), which is assumed to be a perfect square. So suppose that $\xi_{n,l}(A)$ is an essential return, and $A$ is not a partition element according to Definition~\ref{partition}, then partition $A$ into four perfect squares $A_{j_1} \subset A$, $j_1=1,2,3,4$  of equal size. If each of these subsquares satisfies Definition~\ref{partition}, then stop. If not, for each subsquare $A_{j_1} \subset A$ that is not a partition element, partition $A_{j_1}$ into four new subsquares $A_{j_1 j_2}$, $j_2=1,2,3,4$  of equal size. If they are partition elements, then stop. Otherwise go on until all subsquares are partition elements. In this way we obtain a partition of $A$ into subsquares of possibly different sizes of the form $2^{-k}$ times the side length of $A$. We get a collection of dyadic subsquares $A_{n}^{i} \subset A$ such that $A=\cup_i A_{n}^{i}$ and
\[
  \frac{1}{3} \frac{\dist(\xi_{n,l}(A_{n}^{i}),\jrit_{A_n^i})}
  {(\log \dist(\xi_{n,l}(A_{n}^{i}),\jrit_{A_n^i}))^2}\leq\diam\xi_{n,l}(A_{n}^{i})\leq\frac{\dist(\xi_{n,l}(A_{n}^{i}),\jrit_{A_n^i})}{(\log \dist(\xi_{n,l}(A_{n}^{i}),\jrit_{A_n^i}))^2}.
\]
By construction, each $A_{n}^{i}$ is a partition element, as defined in ~\ref{partition} (the constant $1/3$ is chosen because of small distortion; in an completely affine situation, $1/2$ would suffice). At this point we will need to delete parameters which violate the basic assumption. But it turns out that these deleted parameters constitute only a small portion. After (possibly) deleting parameters not satisfying the basic assumption, we continue to iterate each partition element individually.

\section{Large deviations and escape of partition elements}
We now consider a partition element in $\QQ$ and follow it in a time window of the type $[m,(1+\io) m]$, for some (small) $\io > 0$. This section is very similar to \cite{BC2} where the original ideas were developed. See also \cite{MA-Z,MA7}. Let $A_n(a) \subset \QQ$ be a partition element at time $n$, containing the parameter $a$. Since the proofs are very similar to these earlier papers, we are not going through all the proofs here again but instead give references. 

\begin{Def}
We say that $\xi_{n,l}(A_n(a))$ has escaped or is in escape position, if $n$ does not belong to a bound period and $\diam(\xi_{n,l}(A_n(a))) \geq S$ before partitioning. We also speak of escape situation for $A_n(a)$ and say that $A_n(a)$ has escaped if $\xi_{n,l}(A_n(a)) $ has escaped. 
  \end{Def}

  The first observation is that the measure of parameters deleted between two consecutive essential returns is exponentially small in terms of the return time of the former return. See Lemma~8.1 in \cite{MA7}. 

  \begin{Lem} \label{basic-param}
    Let $\xi_{\nu,l}(A)$ be an essential return, $A \subset \EE_{\nu,l}(\ga_I) \cap \BB_{\nu,l}$ and let $\xi_{\nu'}(A)$ be the next essential return. Then if $\hat{A}$ is the set of parameters in $A$ that satisfy the basic assumption at time $\nu'$, we have
    \[
m(\hat{A}) \geq (1-e^{-\al \nu})m(A). 
      \]
    \end{Lem}
    
  \begin{proof}
We first show that $\xi_{\nu,l}(A)$ grows rapidly during the bound period $p$. By Lemmas~\ref{disbound} and \ref{md} and the definition of the bound period, we get, for any $a\in A$,
    \begin{align}
      \diam(\xi_{\nu+p+1,l}(A)) &\sim \frac{e^{-rd_i}}{r^{2}} \vert Df_{a}^p(\xi_{\nu+1,l}(a))\vert \nonumber \\
      &\sim  \frac{\vert \xi_{\nu+p+1,l}(a) - \xi_{p+1,i}(a) \vert}{r^{2}} \nonumber \\
      &\geq C e^{-\al (p+1) - 2 \log r} \dist(\xi_{p+1,i}(a),\jrit_{a}) \nonumber \\
      &\geq C K e^{-2\al (p + 1) - \al (p+1) - 2 \log r} \geq e^{-(7/2) \al p - 2 \log r},
    \end{align}
    if $p$ is large. So, by Lemma~\ref{oel} and Lemma~\ref{length},
    \begin{multline} \label{diameter}
      \diam(\xi_{\nu',l}(A)) \geq \diam(\xi_{\nu+p+1,l}(A)) C' e^{\ga_H(\nu'-(\nu+p+1))} \\
      \geq C' e^{-(7/2) \al p - 2 \log r} 
      \geq  e^{-7 \al d r/\ga - 2 \log r } \geq e^{ - \frac{8 \al d}{\ga} r}.  
    \end{multline}
So by the main distortion Lemma~\ref{md} together with Lemma~\ref{levin}, we see that the measure of parameters deleted at time $\nu'$ is 
    \[
     \frac{ m(A) - m(\hat{A})}{m(A)}  \leq 2 \frac{(e^{-2 \al \nu'})^2}{\diam(\xi_{\nu'}(A))^2} \leq
     2e^{-\al ( 4 - \frac{16 \al d}{\ga} ) \nu} \leq e^{-\al \nu},
        \]
        since $\al \hat{d}/\ga \leq 1/100$ ($\hat{d}$ is the maximal multiplicity of the critical points).  
    \end{proof}

    We next state the following lemma, which is a correspondence to Lemma~8.3 in \cite{MA7}.

    \begin{Lem}
      Let $\xi_{\nu,l}(A)$ be an essential return, $A \subset \EE_{\nu,l}(\ga_I) \cap \BB_{\nu,l}$, with
      $\dist(\xi_{\nu,l}(A), \jrit_A) \sim e^{-r}$. If $q$ is the time after this return spent on inessential returns up until $\xi_{n,l}(A_n(a))$ either escapes, or makes an essential return, $n > \nu$, whichever comes first, then
      \[
q \leq \frac{1}{2} h r,
\]
where $h = 8\hat{d}^2/r$. 
      \end{Lem}

We now assume that we have a partition element $A \subset \QQ$ at time $m$. We follow a parameter $a \in A$ in the time window $[m,(1+\io)m]$, for some (small) $\io > 0$. Suppose that $\nu_0, \nu_1, \nu_2, \ldots, \nu_s$ are the essential returns in this time window for $a$. In addition we assume that $a \in \EE_{k,l}(\ga_I)$ for  $k\leq (1+\io)m$ (this will be satisfied a posteriori). At each return $\nu_j$ the basic approach rate assumption may  force us to delete a fraction of parameters.

Let now $A_j = A_{\nu_j}(a)$ and suppose that $\dist(\xi_{\nu_j,l}(A_j), \jrit_{A_j}) \sim e^{-r_j}$. Then by \eqref{diameter},  we have, 
\[
\frac{m(A_{j+1})}{m(A_j)} \leq C \frac{(e^{-r_{j+1}})^2}{(e^{-8 \al \hat{d} r_j / \ga_I})^2} = C \frac{e^{-2r_{j+1}}}{e^{-16 \al \hat{d} r_j / \ga_I}}.
  \]
  So if we look at the sequence of parameter squares, the measure of $A_s$ compared to $A_1$ is
  \[
\frac{m(A_s)}{m(A_0)}  = \prod_{j=0}^{s-1} \frac{m(A_{j+1})}{m(A_j)} \leq C^s \prod_{j=1}^{s-1} \frac{e^{-2r_{j+1}}}{e^{-16 \al \hat{d} r_j/\ga_I}} .
  \]

  Now write $R = r_1 + \ldots + r_s$. Then, putting $r_0 = r$, we have
  \[
\frac{m(A_s)}{m(A_0)} \leq C^s e^{r_0 16 \al \hat{d}/\ga_I - \sum_{j=1}^{s-1} r_j (2 - 16 \al \hat{d}/\ga_I) - r_s} = C^s e^{r_0 16 \al \hat{d}/\ga_I - (3/2) R}.
\]
So we suppose that $\xi_{\nu_0,l}(A_0)$ is an essential return and will estimate the measure of parameters that do not escape after a long time. If the parameter $a \in A=A_0$ has $s$ essential returns before it has escaped then, with $\nu_0=\nu$, 
\[
E_l(a,\nu) \leq p+1 + \sum_{j=0}^s (h/2) r_j \leq hr + hR,
\]
where we have included the first bound period $p$, which is bounded by $2 \al \hat{d}/\ga_I < \hat{d}^2/\ga_I$ in $hr$.

The number of combinations of $r_j$'s such that $R=r_1 + \ldots + r_s$ is at most 
\[
\binom{R+s-1}{s-1}. 
\]
Since $\xi_{n,l}(A)$ is almost a small perfect square by the strong distortion lemma, there are maximum about $2\pi e^{-r}/r^2$ number of such disjoint squares at distance $e^{-r}$ from the critical points, if $\diam(\xi_{n,l}(A)) \sim e^{-r}/r^2$. Note also that $s\De \leq R$. Let $s \De = q R$, for some $0 < q \leq 1$. 
Taking this into account we get by Stirling's formula that the number of combinations is 
\begin{align}
  \binom{R+s-1}{s-1} &\leq C \frac{(R+s-1)^{R+s-1} e^{-R-s+1}}{R^R (s-1)^{s-1} e^{-R} e^{-s+1}} \sqrt{\frac{R(s-1)}{R+s-1}} \\
&\leq C \biggl( \frac{(1+(q /\De))^{1+q/\De}}{(q/\De)^{q/\De}} \biggr)^R \sqrt{R} \\
&\leq  e^{R/32} (1+\eta(\De))^R, 
\end{align}
where $\eta(\De) \raw 0$ as $\De \raw \infty$, for $R$ large enough (i.e., $\De$ large enough).
Continuing following the the earlier papers, we let $A_{s,R} \subset A$ be the set of all parameters which have exactly $s$ essential return in the time window $[m,(1+\io) m]$, for some $\io > 0$ and fixed $R$. If we let $R$ and $s$ vary, we get a partition of $A$ into countably many subsquares. For fixed $s$ and $R$, let $\hat{A}_s$ be the largest of all such subsquares. Then we get,
\[
m(A_{s,R}) \leq m(\hat{A}_s) e^{R/32} (1 + \eta(\De))^R. 
  \]

Now, we go through the same type of calculations as in \cite{MA7} et al.

  \begin{align}
m(\{ a \in  A : E_l(a,\nu) &= t \}) \\ &\leq \sum_{R \geq t/h - r_0, s \leq R/\De} m(A_{s,R})  \nonumber \\
  &\leq \sum_{R \geq t/h - r_0, s \leq R/\De} m(\hat{A}_s) e^{R/32}(1 + \eta(\De))^R \nonumber \\
                                  &\leq m(A) \sum_{R=t/h - r_0}^{\infty} \sum_{s=1}^{R/\De} e^{R/32}(1 + \eta(\De))^R C^s e^{r_0 (16 \al \hat{d} /\ga) - (3/2) R } \nonumber \\
                           &\leq  C' m(A) \sum_{R=t/h - r_0}^{\infty} C^{R/\De} e^{R/32 + R \log(1 + \eta(\De)) -(3/2) R+ (16 \hat{d} \al /\ga)r_0} \nonumber \\
                                  &\leq  C' m(A) e^{-(\frac{t}{h} - r_0)\frac{46}{32}  + (16 \hat{d} \al /\ga)r_0} \nonumber \\
  &\leq C' m(A) e^{-\frac{t}{h}\frac{46}{32} + (\frac{46}{32} + \frac{16 \hat{d}\al}{\ga}) r_0}.
\end{align}
for some constant $C' > 0$. 

By the condition on $\al$, if $\ga \geq \ga_I$, we get an estimate of the measure of parameters for large escape times. Let us suppose that $t > 2hr_0$. Then
\begin{equation} \label{long-escapes}
 m(\{ a \in A : E_l(a,\nu) = t \}) \leq Ce^{-\frac{t}{3h}} m(A). 
\end{equation}

Of course we may put
\begin{equation} \label{long-escapes2}
 m(\{ a \in A : E_l(a,\nu) \geq t \}) \leq Ce^{-\frac{t}{3h}} m(A). 
\end{equation}
for possibly another constant $C >0$. 

\section{Conclusion and proof}  

Choose $\vep_0 > 0$ and consider a $\vep_0$-neighbourhood $\NN_{\vep_0}$ of the Julia set $\JJ(f_0)$. Then $\widehat{\C} \sm \NN_{\vep_0}$ is a compact subset of the Fatou set. Hence there is $\vep > 0$ such that $\JJ(f_a) \in \NN_{\vep_0}$ holds for all $a \in\QQ$. Consequently, $\FF(f_a) \supset \widehat{\C} \sm \NN_{\vep_0}$, for $a \in \QQ$.

Now suppose that $\xi_{n,l}(A)$ is in escape position, i.e. has diameter comparable to $S$. If we choose $\vep_0 \ll S$, then by the strong distortion control:
\begin{equation} \label{hyp-density}
  m(\{ a \in A : \xi_{n,l}(a) \in \FF(f_a) \}) \geq m(A) (1-\vep_0'),
  \end{equation}
where $\vep_0' \raw 0$ as $\vep_0 \raw 0$. 

For the $\vep > 0$ chosen from the beginning, let $\al$ be such that $32 \hat{d}^2 \al/\ga_I \leq \io/2$. Then, given a first essential return $\xi_{\nu_0,l}(A)$ with $\dist(\xi_{\nu_0,l}(A),\jrit_A) \sim e^{-r}$, we have that $2hr \leq 4h \al n = 32 \hat{d}^2 \al /\ga_I \leq \io n$. According to \eqref{long-escapes2}, parameters in $A$ that have escape time longer than $\io n$ are very few in measure, i.e. less than $Ce^{-2hr/3h} m(A) \leq e^{-r/2} m(A) < \vep' m(A)$, for some $\vep' > 0$, for $r \geq \De$ large enough. Let us disregard from them.   The rest of the parameters makes escape before $n + \io n $ and we can use the estimate  \eqref{hyp-density}, given that the Lyapunov exponent does not drop below $\ga_I$. But since $\xi_{\nu_0,l}(A)$ is a first essential return, we have that all $a \in A$ have $a \in \EE_{\nu_0,l}(\ga_B)$, so that, at time $(1 + \io) \nu$ we have, given that $a \in \BB_{(1+\io)\nu, l}$, that indeed $a \in \EE_{(1+\io)\nu, l}(\ga_I)$ (see the definitions of $\ga_B$ and $\ga_I$).


Now we know from Lemma~\ref{startup} that, for each $c_l$ there is an $N_l > 0$ such that $\xi_{N_l,l}(\QQ)$ satisfies the statements in Lemma~\ref{startup}, i.e. bounded distortion of $\xi_{N_l,l}(a)$ on $\QQ$ and that $\xi_{N_l,l}(\QQ)$ is an essential return or escape situation. If it is an escape situation we are done, and can use \eqref{hyp-density}. Suppose that $N_1 = \min \{ N_l \}$ and $N_2 = \max \{ N_l \}$. To be able to use the binding information for all critical points $c_l(a)$ for $a \in \QQ$ we need to make sure that the bound periods for returns in a time window of the type $(N_l,(1+\io )N_l)$, where $\iota > 0$ is given above, are all smaller than $N_1$.  Actually, we make $\al$ so small that all such bound periods satisfy
\[
  p \leq \frac{4 \hat{d} \al }{\ga_I} (1+\io)N_2 \leq N_1.
\]
By doing this we can use the binding information of all critical points as said. We also delete parameters not satisfying the basic assumption. Using Lemma~\ref{basic-param}, from this we conclude that, for every $\vep' > 0$ (depending on $\io$ and $\al$), we get sets $\Om_l \subset \QQ$ of measure
\[
m(\Om_l) \geq (1-\vep') (1-e^{- \al N_l}) m(\QQ) 
  \]
such that every partition element in $A$ in $\Om_l$ has escaped and 
\[\Om_l \subset \EE_{(1+\io)N_2,l}(\ga_I) \cap \BB_{(1+\io)N_2,l}.
\] 
Moreover, from \eqref{hyp-density} we get that 
\begin{align} \label{hyp-density2}
  m(\{ a \in \QQ : \xi_{n,l}(a) \in \FF(f_a) \}) &\geq m(\Om_l) (1-\vep_0') \\
  &\geq (1-\vep_0')(1-\vep')(1-e^{- \al N_l}) \geq m(\QQ) (1-\vep''), 
\end{align}
for some $\vep'' > 0$ arbitrarily small. Taking the intersection of all critical points and noting that $f_a$ is hyperbolic if all critical points belong to the Fatou set, we get, where $d'$ are the number of critical points, 
\[
m(\{ a \in \QQ : \text{ $f_a$ is hyperbolic }\}) \geq m(\QQ) (1-d' \vep'') .
  \]

Since the set of degenerate one-dimensional families in the parameter space $\la_{d,\oli{p'}}$ of rational maps around $f_0$ has measure zero, we get by Fubini's theorem that $f_0$ is a Lebesgue density point of hyperbolic maps in $\La_{d,\oli{p'}}$.

\bigskip

\bibliographystyle{plain}
\bibliography{ref}

\bigskip
\emph{Centre for Mathematical Sciences}

\emph{Lund University}

\emph{Box 118, 22 100 Lund, Sweden}
 
\medskip
\emph{magnus.aspenberg@math.lth.se}

\smallskip
\emph{mats.bylund@math.lth.se}

\smallskip
\emph{weiwei.cui@math.lth.se}

\end{document}